\documentclass[a4paper,10pt,reqno]{amsart}

\usepackage{quiver}
\usepackage[utf8]{inputenc}
\usepackage[T1]{fontenc}
\usepackage{amsthm}
\usepackage{amsmath}
\usepackage{amssymb}
\usepackage[inline]{enumitem}
\usepackage{bbold}
\usepackage{hyperref}
\usepackage{fancyhdr}
\usepackage{mathrsfs}
\usepackage{stmaryrd}
\usepackage{tikz-cd}
\usetikzlibrary{arrows}
\usepackage{scalerel}
\usepackage{array}
\usepackage[all,cmtip]{xy}
\usepackage{lipsum}

\theoremstyle{definition}
\newtheorem{theorem}{Theorem}[section]
\newtheorem{proposition}[theorem]{Proposition}
\newtheorem{lemma}[theorem]{Lemma}

\newtheorem{definition}[theorem]{Definition}
\theoremstyle{remark}
\newtheorem{remark}[theorem]{Remark}
\newtheorem{example}[theorem]{Example}

\newcommand{\AtoMon}{\mathsf{AtoMon}}
\newcommand{\RedAtoMon}{\mathsf{RedAtoMon}}
\newcommand{\Mon}{\mathsf{Mon}}
\newcommand{\Grp}{\mathsf{Grp}}
\newcommand{\Set}{\mathsf{Set}}
\newcommand{\Quiv}{\mathsf{Quiv}}
\newcommand{\A}{\mathscr A}
\newcommand{\At}{\mathsf A}
\newcommand{\F}{\mathsf F}

\newcommand{\id}{\operatorname{id}}
\newcommand{\Hom}{\operatorname{hom}}

\newcommand{\Cal}[1]{\mathcal{#1}}

\title[Categorical Algebra of Atomic Monoids]
{Categorical Algebra of Atomic Monoids:\\ Presentability, Regularity, and Pretorsion Theories}

\author{Federico Campanini}
\address{(F.~Campanini) Institut de Recherche en Math\'ematique et Physique, Universit\'e Catholique de Louvain | Chemin du cyclotron 2, 1348 Louvain-la-Neuve, Belgium}
\email{federico.campanini@uclouvain.be}

\author{Laura Cossu}
\address{(L.~Cossu) Dipartimento di Matematica e Informatica, Università degli Studi di Cagliari | Palazzo delle Scienze, Via Ospedale 72, 09124 Cagliari (Italy)}
\email{laura.cossu3@unica.it}

\thanks{This manuscript was written during a research visit by L.~Cossu to the Université catholique de Louvain in July 2026. The visit was supported by the \textit{Programma Mobilità Giovani Ricercatori (MGR)} of the University of Cagliari, funded by the Regione Autonoma della Sardegna under L.R. No.~7 of August 7, 2007. L.~Cossu gratefully acknowledges this support. She also acknowledges the support of the INdAM-GNSAGA research group.
The authors are grateful to Fosco Loregian and Salvatore Tringali for preliminary conversations that inspired the development of this paper.}

\begin{document}

\begin{abstract}
We study the category $\AtoMon$ of atomic monoids and atom-preserving homomorphisms. We prove that $\AtoMon$ is locally finitely presentable by exhibiting a strong generator consisting of compact objects. We show that $\AtoMon$ admits (regular epi, mono)-factorizations but that it is not a regular category: we construct a regular epimorphism which is not pullback-stable. We also establish adjunctions for the group of units and explicitly construct the ``cofree atomic monoid'' over an arbitrary monoid. Finally, we exhibit a way to lift torsion theories of $\Grp$ to pretorsion theories of $\AtoMon$ and extend this construction to a more general setting.
\end{abstract}

\maketitle

\section{Introduction}

Let $H$ be a monoid. A non-unit $a\in H$ is an \emph{atom} if it cannot be written as a product of two non-units, and $H$ is \emph{atomic} if every non-unit is a finite product of atoms. Factorization theory studies factorizations into atoms and the arithmetic invariants that measure their possible non-uniqueness. We refer to the monograph~\cite{GeHK06} for deep insight into this subject.

The category $\AtoMon$ was introduced in \cite{CCT} to bring categorical methods into this setting. Its objects are atomic monoids, and its morphisms are atom-preserving monoid homomorphisms. One of the main motivations for introducing $\AtoMon$ was to study the arithmetic behavior of universal categorical constructions. In \cite{CCT}, products and coproducts were constructed explicitly \cite[Theorems~4.6 and~5.4]{CCT}, and their sets of factorization lengths, together with related invariants, were computed \cite[Theorems~4.10 and~5.7]{CCT}. These constructions provide flexible ways to produce new atomic monoids while controlling their arithmetic, and thus offer useful tools for addressing realization problems in factorization theory.

More generally, all small limits and colimits in $\AtoMon$ were constructed in \cite{CCT}. It is worth noting that colimits in $\AtoMon$ are computed as in the category $\Mon$ of monoids, whereas limits generally are not~\cite[Theorems~6.3 and~6.6]{CCT}. It turns out that $\AtoMon$ is a complete and cocomplete non-full subcategory of $\Mon$, with a categorical structure that reflects the distinction between units, atoms, and reducible non-units. These results show that $\AtoMon$ is not only a natural setting for factorization-theoretic questions, but also a category with a rich structure of its own. It is therefore natural to ask which other standard properties of categorical algebra it satisfies. The present paper addresses this question.

We first prove in Section~\ref{sec:lfp} that $\AtoMon$ is locally finitely presentable (Theorem~\ref{thm:lfp}) by exhibiting a strong generator consisting of compact objects. We also show that the compact objects of $\AtoMon$ are precisely the atomic monoids that are finitely presented as ordinary monoids (Theorem~\ref{thm:compact-characterization}).

In Section~\ref{sec:regularity}, we investigate the regularity of $\AtoMon$. In Subsection~\ref{subsec:regepimonofact}, we show that every morphism in $\AtoMon$ admits a canonical (regular epi, mono)-factorization, and we characterize regular epimorphisms in terms of their kernel in $\Mon$. However, regular epimorphisms are not stable under pullback: in Subsection~\ref{subsec:regepi_notstable}, we explicitly construct a regular epimorphism whose pullback is not regular. Consequently, $\AtoMon$ is not a regular category and, in particular, is not equivalent to a variety of universal algebras. We conclude this part by considering the regular completion of $\AtoMon$ in the sense of Carboni and Vitale \cite{CarboniVitale}; see Subsection~\ref{subsec:regular-completion}.

In Section~\ref{sec:adjunctions}, we study some natural functors associated with atomic monoids. We prove in Subsection~\ref{subsec:grpunits} that the ``group-of-units functor'' $(-)^\times\colon \AtoMon\to\Grp$ has both a left and a right adjoint. Moreover, we construct in Subsection~\ref{subsec:atomization} the ``atomization functor'' $\At\colon \Mon \to \AtoMon$, and prove that it is the right adjoint to the inclusion functor $U\colon \AtoMon\to\Mon$.

Finally, in Section~\ref{sec:pretorsion}, we exhibit a way to lift torsion theories of $\Grp$ to pretorsion theories of $\AtoMon$. In particular, we show that the pair $(\Grp, \RedAtoMon)$ is a pretorsion theory in $\AtoMon$, where $\RedAtoMon$ is the full subcategory of $\AtoMon$ consisting of atomic monoids with trivial group of units (see Subsection~\ref{subsec:pretorsion_atomon}). From the above construction we deduce in Subsection~\ref{subsec:pretorsion_general} a general criterion to lift pretorsion theories along a coreflection (Theorem~\ref{thm:lifting-torsion-theory}).

\subsection{Notation and preliminaries}\label{subsec:atomic-preliminaries} 
We denote by $\mathbb Z, \mathbb N^+$ and $\mathbb N$ the three sets of integers, positive integers, and non-negative integers, respectively. Throughout the paper, monoids are written multiplicatively, have an identity element, and are not assumed to be commutative. For a monoid $H$, we denote its identity by $1_H$ and its group of units by $H^\times$. If $H^\times=\{1_H\}$, then $H$ is said to be {\em reduced}. A non-unit $a\in H$ is an \emph{atom} if every equality $a=xy$, with $x,y\in H$, implies that $x\in H^\times$ or $y\in H^\times$. The set of atoms of $H$ is denoted by $\A(H)$.
The monoid $H$ is \emph{atomic} if every non-unit can be written as a finite non-empty product of atoms. Every atomic monoid is {\em Dedekind-finite}, that is, a product of two elements is a unit if and only if each factor is a unit \cite[Proposition 2.30]{FT}. We will freely use this fact without further mention.

A homomorphism $f\colon H\to K$ between atomic monoids is \emph{atom-preserving} if
$f\bigl(\A(H)\bigr)\subseteq\A(K)$.
The category $\AtoMon$ has atomic monoids as objects and atom-preserving monoid homomorphisms as morphisms. We recall from \cite[Remark 3.1(ii)]{CCT} that for
$f\colon H\to K$ a morphism in $\AtoMon$:
$$
    x\in H^\times
    \Longleftrightarrow
    f(x)\in K^\times
 \qquad \text{ and }\qquad
    x\in\A(H)
    \Longleftrightarrow
    f(x)\in\A(K).
$$
Consequently, $f$ also preserves and reflects the class $H\setminus\bigl(H^\times\cup\A(H)\bigr)$
of non-units which are not atoms.

A useful consequence is the following composition property.

\begin{lemma}\label{lem:cancel-atom-preservation}
Let $f\colon M \to N$ and $g\colon L \to M$ be monoid homomorphisms between atomic monoids. If $f$ and $fg$ are morphisms in $\AtoMon$, then $g$ is a morphism in $\AtoMon$. In particular, a bijective morphism in $\AtoMon$ is an isomorphism in $\AtoMon$.
\end{lemma}

\begin{proof}
For every $a\in\A(L)$, the element $fg(a)$ is an atom of $N$. As $f$ is a morphism in $\AtoMon$, it reflects atoms, so $g(a)$ is an atom of $M$. The second assertion follows from the fact that if $f$ is a bijective monoid homomorphism, then also its set-theoretic inverse is a monoid homomorphism.
\end{proof}

We will also use the characterization of monomorphisms proved in
\cite[Remark~3.1(iv)]{CCT}: a morphism $f\colon H\to K$ is a monomorphism in $\AtoMon$ if and only if its restriction to $H^\times\cup\A(H)$ is injective. In particular, a monomorphism in $\AtoMon$ need not be injective on the whole underlying monoid.

We refer to \cite[Sections~4-6]{CCT} for the details of the constructions of limits and colimits in $\AtoMon$.

\medskip

Finally, for a set $X$, we denote by $\F(X)$ the {\em free monoid on $X$}. Its elements are finite words in the alphabet $X$, its identity is the empty word, and its atoms are precisely the words of length one, which we identify with the elements of $X$. For a monoid $H$, we write $H=\langle X\mid \Cal R\rangle$ if $H$ admits a presentation in terms of a set $X$ of generators and a set $\Cal R$ of relations on $\F(X)$. To ease the notation, we will write elements of $H=\langle X\mid R\rangle$ as words in $\F(X)$, implicitly regarding them as representatives of equivalence classes in the quotient of $\F(X)$ by the smallest monoid congruence containing the relations in $\Cal R$.
A monoid is \emph{finitely presented} if it admits a presentation where both $X$ and $R$ are finite. Moreover, we denote by $\langle X\mid\mathcal R\rangle_{\mathrm{ab}}$ the commutative monoid with set of generators $X$ and defining
relations $\mathcal R$.

\section{Locally finite presentability}\label{sec:lfp}
Recall that the category $\Mon$ is locally finitely presentable~\cite[Example~1.2(5)]{AR}. It is natural to ask whether $\AtoMon$ satisfies the same property. Local finite presentability is not automatically inherited by subcategories, so the question requires a direct argument. This is the purpose of this section.

In what follows, we use the terminology \emph{compact object} for a finitely presentable
object in the categorical sense, reserving \emph{finitely presented
monoid} for a monoid admitting a presentation with finitely many
generators and finitely many relations. We start by recalling the key definitions.

\begin{definition}\cite[Chapter~1]{AR}\label{def:locfinpres}
Let $\mathcal C$ be a locally small category with filtered colimits.  An
object $X\in\mathcal C$ is \emph{compact} if the functor
$$
  \Hom_{\mathcal C}(X,-):\mathcal C\to\Set
$$
preserves filtered colimits. The category $\mathcal{C}$ is said to be {\em locally finitely presentable} if it is cocomplete and there exists a set $\mathcal{K}$ of compact objects such that every object of $\mathcal{C}$ is a filtered colimit of objects from $\mathcal{K}$.
\end{definition}

\begin{lemma}\label{lem:fp-monoid-compact}
A finitely presented atomic monoid is compact in $\AtoMon$.
\end{lemma}

\begin{proof}
Let $H$ be finitely presented as a monoid and let
$D:I\to\AtoMon$ be a filtered diagram.  Colimits in $\AtoMon$ are
computed as in $\Mon$ by \cite[Theorem~6.3]{CCT}.  Since $H$ is compact
in $\Mon$ (see \cite[Corollary~3.13]{AR}), every monoid homomorphism
$$
  f:H\to\mathop{\mathrm{colim}}_{i\in I}D_i
$$
factors through some colimit morphism $d_i:D_i\to\mathrm{colim}D$.
If $f$ is a morphism in $\AtoMon$, then the resulting factor
$g:H\to D_i$ is atom-preserving by Lemma~\ref{lem:cancel-atom-preservation}.  The uniqueness follows from compactness in $\Mon$.  Thus $H$ is compact in $\AtoMon$.
\end{proof}

Recall that a set $\mathcal G$ of objects of a category $\mathcal C$ is called a {\em generator (for $\mathcal C$)} if for each pair $f_1, f_2\colon K\to K'$ of distinct morphisms there exists an object $G \in \mathcal G$ and a morphism $g \colon G\to K$ with $f_1g\neq f_2 g$.
A generator $\mathcal G$ is called {\em strong} provided that for each object $K$ and each proper subobject of $K$ there exists a morphism $G\to K$ with $G \in \mathcal G$ which does not factor through that subobject (see, for instance, \cite{AR}). 

According to \cite[Theorem~1.11]{AR}, a category $\mathcal C$ is locally finitely presentable if and only if it is cocomplete and has a strong generator consisting of compact objects. We use this characterization to prove that $\AtoMon$ is locally finitely presentable.

\medskip

For every finite set $X$, let $\F(X)$ be the free monoid generated by $X$. If
$\omega_1,\omega_2\in\F(X)$ are distinct words of length at least two, set
$$
  M(X;\omega_1,\omega_2):=\langle X\mid \omega_1=\omega_2\rangle.
$$

\begin{lemma}\label{lem:one-relator-atomic}
For every finite set $X=\{x_1,\dots, x_n\}$ and every pair of distinct words
$\omega_1,\omega_2\in\F(X)$ with $|\omega_1|,|\omega_2|\geq 2$, the monoid $M(X;\omega_1,\omega_2)$ is reduced
and atomic, and
$$
  \A(M(X;\omega_1,\omega_2))=\{ x_1,\ldots,x_n\}.
$$
In particular, $M(X;\omega_1,\omega_2)$ is compact in $\AtoMon$.
\end{lemma}

\begin{proof}
Every elementary application of the relation $\omega_1=\omega_2$ replaces a subword of length at least two by a word of length at least two.
Consequently, no non-empty word is congruent to the empty word, and no word of length one is congruent to a different word.  Thus the quotient is reduced and the elements $x_1,\ldots,x_n$ are
pairwise distinct atoms.  Every other non-identity element has a
representative of length at least two and is therefore a product of two non-units.  Hence, these are all the atoms and the monoid is atomic. The last assertion follows from Lemma~\ref{lem:fp-monoid-compact}.
\end{proof}

Let $F:=\F(\{x\})$ be the free monoid on one generator, and let
$G:=\langle g,g^{-1}\mid gg^{-1}=g^{-1}g=1\rangle$ be the infinite
cyclic group, regarded as an object of $\AtoMon$. For every positive integer $n$, fix a finite set $X_n:=\{x_1,\dots, x_n\}$. Consider the set
\begin{equation}\label{eq:G}
\mathcal G:=\{F,G\}\cup
\left\{
  M(X_n;\omega_1,\omega_2) :    n \in \mathbb N^+,\ \omega_1,\omega_2\in\F(X_n),\ \omega_1\neq\omega_2,
    |\omega_1|\geq 2,\ |\omega_2|\geq 2
\right\}.
\end{equation}

\begin{proposition}\label{prop:strong-generators}
The set $\mathcal G$ is a strong generator for $\AtoMon$.
\end{proposition}

\begin{proof} Notice that $\mathcal G$ is clearly a generator for $\AtoMon$ (because so is the set $\{F,G\}$).
Let $m:H\to K$ be a monomorphism in $\AtoMon$, and suppose that every morphism $M\to K$, with $M\in\mathcal G$, factors through~$m$. We prove that $m$ is an isomorphism.

For every $a\in\A(K)$, the morphism $F\to K$ determined by $x\mapsto a$ factors through $m$; hence every atom of $K$ belongs to $m(H)$.  For every $u\in K^\times$, the morphism $G\to K$ determined by $g\mapsto u$
factors through $m$; hence every unit of $K$ belongs to $m(H)$.  Since $K$ is atomic, it follows that $m$ is surjective as a map of sets.

Suppose that $m$ is not injective. Choose $a,b\in H$ with $a\neq b$ and $m(a)=m(b)$. As recalled in Subsection~\ref{subsec:atomic-preliminaries}, a morphism of $\AtoMon$ preserves and reflects the classes of units, atoms, and non-units which are not atoms, while a monomorphism is injective on $H^\times\cup\A(H)$.  Thus $a$ and $b$ are non-units which are not atoms. Choose atomic factorizations
$$
  a=a_1\cdots a_r,
  \qquad
  b=b_1\cdots b_s,
  \qquad r,s\geq 2.
$$
Let $C\subseteq\A(H)$ be the finite set of atoms occurring in these factorizations, choose a set $X=\{x_c:c\in C\}$ of formal symbols, and let
$$
  \omega_1=x_{a_1}\cdots x_{a_r},
  \qquad
  \omega_2=x_{b_1}\cdots x_{b_s}
$$
be two words in $\F(X)$. Observe that $\omega_1$ and $\omega_2$ are distinct, since otherwise $a=b$.  Define the monoid
$M:=\langle X\mid \omega_1=\omega_2\rangle$. By Lemma~\ref{lem:one-relator-atomic}, the generators $x_c$ are precisely the atoms of $M$.  The assignment
$$
  x_c\mapsto m(c)
$$
defines a morphism $\varphi:M\to K$, because $m(a)=m(b)$. Notice that $M$ is isomorphic to a monoid $M' \in \Cal G$, therefore $\varphi=m\psi$ for some morphism $\psi:M\to H$. For every $c\in C$, $m(\psi(x_c))=m(c)$; since both $\psi(x_c)$ and $c$ are atoms, and $m$ is injective on atoms, then $\psi(x_c)=c$.  Applying $\psi$ to the relation $\omega_1=\omega_2$ yields $a=b$, which is a contradiction.  Therefore $m$ is injective.

We then conclude that $m$ is a bijective monoid homomorphism, its inverse preserves atoms because $m$ reflects atoms, and hence $m$ is an isomorphism in $\AtoMon$.
\end{proof}

\begin{remark}
The two-object set $\{F,G\}$ is a generator but it is not strong. Intuitively, it
detects atoms and units, but not relations between products of atoms.
For instance, for the non-injective monomorphism
$$
  m\colon \F(\{\alpha,\beta\})\to
  K:=\langle a,b\mid a^3=b^2\rangle,
  \qquad \alpha\mapsto a,\quad\beta\mapsto b,
$$
every morphism in $\AtoMon$ from either $F$ or $G$ to $K$ factors through $m$.
\end{remark}

\begin{theorem}\label{thm:lfp}
The category $\AtoMon$ is locally finitely presentable.
\end{theorem}

\begin{proof}
The category $\AtoMon$ is cocomplete by \cite[Theorem~6.3]{CCT}.
Every object of $\mathcal G$ is compact by
Lemma~\ref{lem:fp-monoid-compact} and
Lemma~\ref{lem:one-relator-atomic}, and $\mathcal G$ is a strong generator by Proposition~\ref{prop:strong-generators}.  The claim now follows from the strong generator criterion for locally finitely presentable categories recalled above.
\end{proof}

The following result complements Lemma~\ref{lem:fp-monoid-compact}, showing that the compact objects in $\AtoMon$ are precisely the finitely presented (atomic) monoids.

\begin{theorem}\label{thm:compact-characterization}
An atomic monoid $H$ is compact in $\AtoMon$ if and only if it is finitely presented as a monoid.
\end{theorem}

\begin{proof}
The ``if'' direction is exactly Lemma~\ref{lem:fp-monoid-compact}.
For the converse, let $\mathcal G$ be the strong generator defined in \eqref{eq:G}, and let $\mathcal G_{\mathrm{fin}}$ denote its closure under finite colimits in $\AtoMon$.  By the proof of~\cite[Theorem~1.11]{AR}, every object $H$ of $\AtoMon$ can be obtained as a filtered colimit of objects of $\mathcal G_{\mathrm{fin}}$, that is, $H=\mathop{\mathrm{colim}}_{i\in I}D_i$ for a filtered diagram $D\colon I \to \mathcal{G}_\mathrm{fin}$.
Therefore, if $H$ is compact, the identity morphism of $H$ factors through some colimit morphism $D_j \to H$, $j \in I$.  Thus $H$ is a retract of $D_j$ in $\AtoMon$, and hence also in~$\Mon$.

Recall that compact objects in $\Mon$ are precisely the finitely presented monoids~\cite[Example~1.2(5)]{AR}; therefore every object of $\mathcal G$ is compact in $\Mon$. By the definition of $\mathcal G_\mathrm{fin}$, $D_j$ is a finite colimit in $\AtoMon$ (hence in $\Mon$, by~\cite[Theorem~6.3]{CCT}) of compact objects of $\Mon$. Since finite colimits of compact objects are compact~\cite[Proposition~1.3]{AR}, $D_j$ is compact in $\Mon$. Compact objects are closed under retracts~\cite[Remark,~p.13]{AR}, so $H$ is compact in $\Mon$, hence a finitely presented monoid.
\end{proof}

\section{Failure of regularity}\label{sec:regularity}

Let $\Cal C$ be a category. Following \cite{Gran}, a morphism in $\Cal C$ is called a \emph{regular epimorphism} if it is the coequalizer of two parallel arrows. A morphism $f \colon X \to Y$ admits a \emph{(regular epi, mono)-factorization} if there exist a regular epimorphism $p \colon X \to I$ and a monomorphism $m \colon I \to Y$ such that $f=m p$. A category with finite limits is \emph{regular} if (1) every morphism admits a (regular epi, mono)-factorization and (2) the pullback of a regular epimorphism along any morphism is again a regular epimorphism. Examples of regular categories are the category of sets, the category of (abelian) groups, and more generally any variety of universal algebras. In particular, the category $\Mon$ of monoids is regular.

The aim of this section is to address the natural question of whether $\AtoMon$ is also a regular category. We show that morphisms in $\AtoMon$ admit (regular epi, mono)-factorizations. Nevertheless, we construct an explicit example in which the pullback of a regular epimorphism is not regular, showing that $\AtoMon$ is not a regular category. Its regular completion will be discussed in the last subsection.

\medskip

\subsection{Canonical (regular epi, mono)-factorization}\label{subsec:regepimonofact}

Let $f\colon H \to K$ be a morphism in $\AtoMon$. We first want to understand how to identify elements with the same image under $f$. We can of course consider the congruence on $H$ given by $x \sim y$ if and only if $f(x)=f(y)$, which corresponds to the ordinary kernel congruence in $\Mon$
$$
\ker_{\Mon}(f):=\{(x,y) \in H\times H \mid f(x)=f(y)\}.
$$
However, this monoid may fail to be atomic, as we will show in the following example.

\begin{example}\label{ex:non-atomic-monoid-kernel}
Consider the commutative monoids
$$
H:=
\left\langle
a,b,p,q,r,s
\ |\
pq=apq,\; rs=brs
\right\rangle_{\mathrm{ab}}
$$
and
$$
K:=
\left\langle
c,p,q,r,s
\ |\
pq=cpq,\; rs=crs,\; pq=rs
\right\rangle_{\mathrm{ab}}.
$$
Both monoids are reduced and atomic, and their atoms are precisely the generators. Indeed, no defining relation involves the empty word or a word of length one. Define a monoid homomorphism $f:H\to K$ by
$$
f(a)=f(b)=c,\; f(p)=p,\; f(q)=q,\; f(r)=r,\; f(s)=s.
$$
The map is well defined and atom-preserving. Since $pq=rs$ in $K$, the element
$$
\xi:=(pq,rs)
$$
belongs to $\ker_{\Mon}(f)$. Moreover, in $H$ we have $pq=apq$ and $rs=brs$, and hence
$$
\xi=(a,b)\xi.
$$
Since both $(a,b)$ and $\xi$ are non-units of $\ker_{\Mon}(f)$,
the element $\xi$ is not an atom.

We claim that $\xi$ does not admit a factorization into atoms of
$\ker_{\Mon}(f)$. Suppose that
$$
\xi=(u_1,v_1)\cdots(u_n,v_n)
$$
is a factorization into non-units of $\ker_{\Mon}(f)$. Then
$$
u_1\cdots u_n=pq,
\quad
v_1\cdots v_n=rs,
\quad
f(u_i)=f(v_i)
\quad\text{for every }i=1,\ldots,n.
$$
Observe that the defining relations of $H$ do not affect the number
of occurrences of the letters $p,q,r,s$. Thus, $p$ and $q$ occur
exactly once in the product $u_1\cdots u_n$, while $r,s$, and
consequently $b$, do not occur. Analogously, $r$ and $s$ occur
exactly once in the product $v_1\cdots v_n$, while $p,q$, and
consequently $a$, do not occur.

We show that $p$ and $q$ must occur in the same factor $u_i$. Indeed, the difference between the numbers of occurrences of $p$ and $q$ is preserved by all the defining relations of $K$. Since $f(v_i)$ has a representative containing neither $p$ nor $q$, this difference is zero for $f(v_i)$, and hence also for $f(u_i)$. Therefore, if $p$ occurs in $u_i$, then $q$ must occur in the same factor.

Consequently, there exists a unique index $i_0$ such that $u_{i_0}=a^kpq$ for some $k\in\mathbb N$, while $u_i=a^{k_i}$ for some $k_i\in\mathbb N^+$ whenever $i\neq i_0$.

An analogous argument shows that there exists a unique index $j_0$ such that $v_{j_0}=b^hrs$ for some $h\in\mathbb N$, while $v_j=b^{h_j}$ for some $h_j\in\mathbb N^+$ whenever $j\neq j_0$.

We must have $i_0=j_0$. Indeed, if $i_0\neq j_0$, then
$$
f(u_{i_0})=c^kpq
\qquad\text{and}\qquad
f(v_{i_0})=c^{h_{i_0}}.
$$
These two elements are distinct in $K$, since the defining relations
of $K$ preserve the total number of occurrences of the letters
$p,q,r,s$. This contradicts $f(u_{i_0})=f(v_{i_0})$. Therefore, $i_0=j_0$. Since $a^kpq=pq$ and $b^hrs=rs$ in $H$, we obtain
$$
(u_{i_0},v_{i_0})=(pq,rs)=\xi.
$$
Hence every factorization of $\xi$ into non-units contains $\xi$
itself as one of its factors. Since $\xi$ is not an atom, it cannot
be written as a finite product of atoms. Therefore
$\ker_{\Mon}(f)$ is not atomic.
\end{example}

In order to take a congruence internal to $\AtoMon$, let us consider the kernel pair $(\ker(f), r_1, r_2)$ of $f\colon H \to K$. By \cite[Theorem 6.5]{CCT}, $\ker(f)$ is the monoid generated by the set
$$
\Sigma_f:=\{(h_1,h_2) \in H \times H \mid f(h_1)=f(h_2) \text{ and either } h_1, h_2 \in \mathscr A(H) \text{ or } h_1, h_2 \in H^\times\} .
$$

Therefore, $(x,y) \in \ker(f)$ if and only if either $x,y \in H^\times$ and $f(x)=f(y)$, or there exist $n \in \mathbb N^+$ and $a_1, \dots a_n, b_1, \dots, b_n \in \mathscr A(H)$ such that $x=a_1\cdots a_n$, $y=b_1\cdots b_n$ and $f(a_i)=f(b_i)$ for all $i \in \{ 1,\dots,n\}$. Note that $\ker(f)$ is not an ordinary equivalence relation on $H$ in general, as the following example shows.

\begin{example}
    Consider the two atomic monoids
    $$
    H:=\langle a_1,a_2, b_1,b_2,c_1,c_2,c_3,d_1,d_2,d_3 \mid b_1b_2=c_1c_2c_3\rangle
     \quad \text{and} \quad K:=\langle p,q\mid p^2=q^3\rangle.
    $$
    Define a morphism
    $f\colon H \to K$ by
    $$
    f(a_i)=f(b_i)=p \quad \text{ for } i=1,2 \quad  \text{and} \quad f(c_j)=f(d_j)=q \quad \text{for } j=1,2,3.
    $$
    Set $x:=a_1a_2$, $y:=b_1b_2=c_1c_2c_3$ and $z:=d_1d_2d_3$. Then, $(x,y),(y,z) \in \ker(f)$ but $(x,z)\notin\ker(f)$.
\end{example}

Notice that this does not contradict the fact that $\ker(f)$ is an internal equivalence relation in $\AtoMon$. The ``(internal) transitivity morphism'' is defined on the pullback $\ker(f)\times_H\ker(f)$ of $r_1$ and $r_2$ in $\AtoMon$, which is not the monoid of all ``compatible pairs'' $\{((x_1,x_2),(x_3,x_4))\in H^2\times H^2 \mid x_2=x_3\}$.

\medskip

For a morphism $f\colon H\to K$ in $\AtoMon$, let $\theta_f$ be the smallest monoid congruence containing $\Sigma_f$. The coequalizer of $r_1$ and $r_2$ is therefore the projection $p_f\colon H \to H/\theta_f$ (see \cite[Proposition~6.1]{CCT}). We have a factorization of $f$ in $\AtoMon$ of the form
\begin{equation}\label{eq:epimonofact}
    \xymatrix{
    \ker(f) \ar@<0.6ex>[r]^-{r_1} \ar@<-0.6ex>[r]_-{r_2} & H \ar[rr]^f \ar[rd]_{p_f} & & K \\
    & & H/\theta_f \ar[ru]_{m_f} & \; ,
    }
\end{equation}
where $m_f\colon H/\theta_f \to K$ is the map sending, for every $x \in H$, the congruence class $[x]$ to $f(x)$. Notice that the map $m_f$ is a monomorphism in $\AtoMon$ as it is injective on units and atoms of $H/\theta_f$ (but not an injective map in general). Indeed, if $[x]$ and $[y]$ are both atoms or both units and $m_f([x])=m_f([y])$, then $f(x)=f(y)$. Hence $(x,y)\in\Sigma_f\subseteq\theta_f$, and therefore $[x]=[y]$ (cf. the proof of~\cite[Proposition~6.1]{CCT}). This argument proves the following.

\begin{proposition}\label{prop:image-factorization}
Every morphism $f:H\to K$ in $\AtoMon$ admits a (regular epi, mono)-factorization.
\end{proposition}

We conclude this subsection by describing the regular epimorphisms in $\AtoMon$.

Recall that {\em a strong epimorphism} in a category $\Cal C$ is an epimorphism which is left orthogonal to any monomorphism in $\Cal C$ (see \cite{Gran}).

\begin{proposition}\label{prop:strong-regular-epi}
For a morphism $f:H\to K$ in $\AtoMon$, let $f=m_fp_f$ be its (regular epi, mono)-facto\-ri\-za\-tion~\eqref{eq:epimonofact}. The following conditions are
equivalent:
\begin{enumerate}[label=\textup{(\arabic*)}]
\item\label{prop:strong-regular-epi 1} $f$ is a regular epimorphism;
\item\label{prop:strong-regular-epi 2} $f$ is a strong epimorphism;
\item\label{prop:strong-regular-epi 3} $m_f:H/\theta_f\to K$ is an isomorphism;
\item\label{prop:strong-regular-epi 4} $f$ is surjective and $\ker_{\Mon}(f)=\theta_f$.
\end{enumerate}
\end{proposition}

\begin{proof}
The implication \ref{prop:strong-regular-epi 1}$\Rightarrow$\ref{prop:strong-regular-epi 2} is valid in any category, see \cite[Proposition 1.8]{Gran}. To prove \ref{prop:strong-regular-epi 2}$\Rightarrow$\ref{prop:strong-regular-epi 3}, consider the commutative square
$$
\begin{matrix}
H&\xrightarrow{p_f}&H/\theta_f\\
{\scriptstyle f}\downarrow&&\downarrow{\scriptstyle m_f}\\
K&\xrightarrow{\id_K}&K.
\end{matrix}
$$
Since $f$ is a strong epimorphism and $m_f$ is a monomorphism, there is
a diagonal morphism $d:K\to H/\theta_f$ such that $ df=p_f,$ and $m_fd=\id_K.$ Thus $m_f$ is both a monomorphism and a split epimorphism, and hence it is an isomorphism. The implication \ref{prop:strong-regular-epi 3}$\Rightarrow$\ref{prop:strong-regular-epi 1} is clear. It remains to prove the equivalence between \ref{prop:strong-regular-epi 3} and
\ref{prop:strong-regular-epi 4}. 

Assume first that $m_f$ is an isomorphism. The quotient map $p_f$ is surjective, then so is $f$. Moreover, the kernel congruence of $p_f$ in $\Mon$ is $\theta_f$, so
$$
 \ker_{\Mon}(f)
 =\ker_{\Mon}(m_fp_f)
 =\ker_{\Mon}(p_f)
 =\theta_f
$$
and \ref{prop:strong-regular-epi 4} holds. Conversely, suppose that $f$ is surjective and
$\ker_{\Mon}(f)=\theta_f$.  We show that the underlying map of $m_f$ is
bijective. For $x,y\in H$, if $m_f([x])=m_f([y])$,
then $f(x)=f(y)$, so $(x,y)\in\ker_{\Mon}(f)=\theta_f$ and therefore $[x]=[y]$. Thus $m_f$ is injective. Given $z\in K$, choose $x\in H$
with $f(x)=z$. Then, $m_f([x])=z$, so $m_f$ is surjective. It follows that $m_f$ is a bijective morphism in $\AtoMon$, so an
isomorphism by Lemma~\ref{lem:cancel-atom-preservation}. This proves \ref{prop:strong-regular-epi 3}.
\end{proof}

\subsection{A regular epimorphism not stable under pullback}\label{subsec:regepi_notstable}

Notice that Proposition~\ref{prop:image-factorization} does not imply that $\AtoMon$ is a regular category. Regularity additionally requires regular epimorphisms to be stable under pullbacks. In this subsection we construct a regular epimorphism whose pullback is a surjective morphism which is not a regular epimorphism, showing that $\AtoMon$ is not regular.

\medskip

Consider the commutative atomic monoids
\begin{align*}
H&:=
\left\langle
 a_1,a_2,a_3,b,c_1,c_2,c_3,d
 \ \middle|\
 a_1b=a_2c_1,\;a_3c_2=c_3d
\right\rangle_{\mathrm{ab}},\\
Q&:=
\left\langle
 a,b,c,d
 \ \middle|\
 ab=ac,\; ac=cd
\right\rangle_{\mathrm{ab}},\\
J&:=
\left\langle
 a,b,c,d
 \ \middle|\
 ab=cd
\right\rangle_{\mathrm{ab}}.
\end{align*}
All the presentations above are
homogeneous and have no defining relation of degree one. Consequently,
the resulting monoids are reduced and atomic, and their atoms are
precisely their generators.

Define morphisms in $\AtoMon$
$$
 p:H\to Q
 \qquad\text{and}\qquad
 j:J\to Q
$$
by
$$
 p(a_i)=a,\qquad
 p(c_i)=c,\qquad
 p(b)=b,\qquad
 p(d)=d
$$
for every $i\in\{1,2,3\}$, and by letting $j$ send each generator of
$J$ to the generator of $Q$ with the same name. Notice that $p$ and $j$ are both well-defined as they preserve the defining relations of $H,J$ and $Q$.

\begin{lemma}\label{lem:monoregepi}
The morphism $j$ is a monomorphism, while $p$ is a regular epimorphism in $\AtoMon$.
\end{lemma} 

\begin{proof}
Since $j$ is injective on the atoms of $J$ and $J$ is a reduced monoid, $j$ is a monomorphism in $\AtoMon$.
To show that $p$ is a regular epimorphism, we use the characterization of Proposition~\ref{prop:strong-regular-epi}. Since $H$ is reduced, the
congruence $\theta_p$ is generated by the set
$$
\Sigma_p=
   \{(a_i,a_j):i,j\in\{1,2,3\}\}
   \cup
   \{(c_i,c_j):i,j\in\{1,2,3\}\} \cup \{(b,b),(d,d), (1,1)\}.
$$
In $H/\theta_p$ set
$$
 A:=[a_1]=[a_2]=[a_3], \quad B:=[b], \quad C:=[c_1]=[c_2]=[c_3], \quad D:=[d].
$$
Then, 
$$
 H/\theta_p=
 \left\langle A,B,C,D\mid AB=AC,\ AC=CD
 \right\rangle_{\mathrm{ab}} \cong Q,
$$
and the isomorphism is induced by $p$. Proposition~\ref{prop:strong-regular-epi}(3) now shows that $p$ is a
regular epimorphism.
\end{proof}

Let us consider the pullback
$$
\begin{matrix}
P&\xrightarrow{\sigma}&H\\
{\scriptstyle \rho}\downarrow&&\downarrow{\scriptstyle p}\\
J&\xrightarrow{j}&Q \ .
\end{matrix}
$$
By \cite[Proposition~6.5]{CCT}, the monoid $P$ is the submonoid of
$H\times J$ generated by the pairs of atoms

 \begin{equation}\label{eq:ABCD}
 A_i:=(a_i,a),\qquad C_i:=(c_i,c)
 \quad \text{for }i\in\{1,2,3\},\qquad 
 B:=(b,b),\qquad D:=(d,d).
 \end{equation}

\begin{lemma}\label{lem:pullback-presentation}
The pullback object \(P\) admits the presentation
$$
P\cong \left\langle A_1,A_2,A_3,B,C_1,C_2,C_3,D \ \middle|\ A_1BA_3C_2=A_2C_1C_3D \right\rangle_{\mathrm{ab}}.
$$
Under this identification,
$$
\rho(A_i)=a,\qquad \rho(B)=b,\qquad \rho(C_i)=c,\qquad \rho(D)=d.
$$
\end{lemma}

\begin{proof}
Set
$$ M:= \left\langle A_1,A_2,A_3,B,C_1,C_2,C_3,D \ \middle|\ A_1BA_3C_2=A_2C_1C_3D \right\rangle_{\mathrm{ab}}. $$ 
Let $$ \Phi:M\to P $$ be the surjective monoid homomorphism determined by the correspondences in \eqref{eq:ABCD}.
The map is well defined, since the relation
\begin{equation}\label{eq:relation-M}
A_1BA_3C_2=A_2C_1C_3D
\end{equation}
is satisfied by the corresponding elements of both $H$ and $J$. We prove that $\Phi$ is injective. We first show how equalities in $H$ and $J$ can be expressed in terms of exponent vectors. Consider two elements
$$
u=a_1^{r_1}a_2^{r_2}a_3^{r_3}b^{r_4} c_1^{r_5}c_2^{r_6}c_3^{r_7}d^{r_8} \quad \text{and} \quad v=a_1^{s_1}a_2^{s_2}a_3^{s_3}b^{s_4} c_1^{s_5}c_2^{s_6}c_3^{s_7}d^{s_8}
$$
in $H$, where all the exponents are non-negative integers, and write
$$
\mathbf r=(r_1,\ldots,r_8), \qquad \mathbf s=(s_1,\ldots,s_8).
$$
If $u=v$ in $H$, then $u$ can be transformed into $v$ by applying finitely many times, in either direction, the relations
$$
a_1b=a_2c_1, \qquad a_3c_2=c_3d.
$$
Consequently, there exist $k,\ell\in\mathbb Z$ such that
$$
\mathbf r = \mathbf s +k(-1,1,0,-1,1,0,0,0) +\ell(0,0,-1,0,0,-1,1,1).
$$
Similarly, given two elements
$$
x=a^{m_1}b^{m_2}c^{m_3}d^{m_4} \quad \text{and} \quad y=a^{n_1}b^{n_2}c^{n_3}d^{n_4}
$$ in $J$, write the vectors of their non-negative exponents as
$$
\mathbf m=(m_1,m_2,m_3,m_4), \qquad \mathbf n=(n_1,n_2,n_3,n_4).
$$
If $x=y$ in $J$, then there exists $t\in\mathbb Z$ such that
$$
\mathbf m=\mathbf n+t(-1,-1,1,1).
$$
Now let $U,V\in M$, represented by
$$
U= A_1^{r_1}A_2^{r_2}A_3^{r_3}B^{r_4} C_1^{r_5}C_2^{r_6}C_3^{r_7}D^{r_8} \quad \text{and} \quad V= A_1^{s_1}A_2^{s_2}A_3^{s_3}B^{s_4} C_1^{s_5}C_2^{s_6}C_3^{s_7}D^{s_8},
$$
where all exponents are non-negative integers, and suppose that $\Phi(U)=\Phi(V)$. Then, $\sigma(\Phi(U))=\sigma(\Phi(V))$ in $H$ and from the previous observation, we obtain
$$
(r_1,\ldots,r_8) = (s_1,\ldots,s_8) +(-k,k,-\ell,-k,k,-\ell,\ell,\ell)
$$
for suitable $k,\ell\in\mathbb Z$.

On the other hand, $\rho(\Phi(U))=\rho(\Phi(V))$ in $J$.
Hence there exists $t\in\mathbb Z$ such that
$$
(r_1+r_2+r_3,\ r_4,\ r_5+r_6+r_7,\ r_8) = (s_1+s_2+s_3,\ s_4,\ s_5+s_6+s_7,\ s_8) +t(-1,-1,1,1).
$$
Combining the two vector equalities, we get $k=\ell=t$. It follows that
$$
(r_1,\ldots,r_8) = (s_1,\ldots,s_8) +k(-1,1,-1,-1,1,-1,1,1).
$$
If $k \geq 0$, then there exists an element $W \in M$ such that
$$
U=W(A_2C_1C_3D)^k \quad \text{and} \quad V=W(A_1BA_3C_2)^k
$$
from which $U=V$. If $k\leq 0$ it suffices to interchange the roles of $U$ and $V$. Thus, $\Phi$ is injective.
\end{proof}

\begin{proposition}\label{prop:not-regular}
The pullback morphism
$$
 \rho:P\to J
$$
is not a regular epimorphism. Consequently, $\AtoMon$ is not a regular
category.
\end{proposition}

\begin{proof}
The morphism $\rho$ is surjective, since its image contains all the generators $a,b,c,d$ of $J$.
Recall that, since $P$ is reduced, the kernel pair $\ker(\rho)$ of $\rho$ is generated by the set
$$
\Sigma_\rho=\{(A_i,A_j):i,j\in\{1,2,3\}\}
 \cup
 \{(C_i,C_j):i,j\in\{1,2,3\}\} \cup \{(B,B), (D,D), (1,1)\}.
$$

Let $\theta_\rho$ be the corresponding monoid congruence. Consider the elements
$$
 x:=A_1B
 \qquad\text{and}\qquad
 y:=C_3D
$$
of $P$. Their images under $\rho$ satisfy
$$
 \rho(x)=ab=cd=\rho(y)
$$
in $J$. Hence
$$
 (x,y)\in\ker_{\Mon}(\rho).
$$

We prove directly that $(x,y)\notin\theta_\rho$.  Quotienting the
presentation of $P$ of Lemma~\ref{lem:pullback-presentation} by $\theta_\rho$ gives
$$
 P/\theta_\rho\cong
 \left\langle
 A,B,C,D
 \ \middle|\
 A^2BC=AC^2D
 \right\rangle_{\mathrm{ab}}.
$$
The classes of $x=A_1B$ and $y=C_3D$ are represented by $AB$ and
$CD$, respectively.  These elements are distinct in the quotient, because the relation only acts on words of length at least four. Hence 
$$
 (x,y)\notin\theta\rho
$$
and by Proposition~\ref{prop:strong-regular-epi}(4), $\rho$ is not a regular epimorphism.
Since $\rho$ is the pullback of the regular epimorphism $p$ along the monomorphism $j$, this shows that regular epimorphisms in $\AtoMon$ are not stable under pullback. Hence $\AtoMon$ is not regular.
\end{proof}

\begin{remark}
Notice that Proposition~\ref{prop:not-regular} implies, in particular, that the category $\AtoMon$ is not Barr-exact, hence it cannot be equivalent to a variety of universal algebras (see \cite{AR,Borceux} for the definitions).
Nevertheless, since $\AtoMon$ is locally (finitely) presentable, from the Gabriel--Ulmer duality (cf. \cite[Chapter~5]{AR}), it follows that $\AtoMon$ is equivalent to the category of models of a finitary finite-limit theory, and hence it is a finitary essentially algebraic category. More precisely, let $\AtoMon_{\mathrm{fp}}$ be a small skeleton of the full subcategory of compact objects of $\AtoMon$. Then the restricted Yoneda functor induces an equivalence
$$
\AtoMon \simeq \operatorname{Lex}\bigl(\AtoMon_{\mathrm{fp}}^{\mathrm{op}},\Set\bigr).
$$
\end{remark}

\subsection{The regular completion}
\label{subsec:regular-completion}

Recall from \cite{CarboniVitale} that any category $\Cal C$ with weak limits can be embedded into a regular category, called the {\em regular completion of $\Cal C$}. We briefly describe this construction for the case of $\AtoMon$. Notice that, since $\AtoMon$ is complete, the description given in \cite[Definition~7]{CarboniVitale} simplifies.

The regular completion $\AtoMon_{\mathrm{reg}}$ of $\AtoMon$ is the category whose objects are the
morphisms of $\AtoMon$. Given two such objects
$$
    f\colon H\to K,
    \qquad
    g\colon L\to M,
$$
let $(\ker(f),r_1,r_2)$ be the kernel pair of $f$ in $\AtoMon$. A morphism from $f$ to $g$
is represented by a morphism in $\AtoMon$
$$
\alpha\colon H\to L \qquad \text{such that}  \qquad g\alpha r_1=g\alpha r_2.
$$
Two representatives $\alpha,\alpha'\colon H\to L$ determine the same
morphism if $g\alpha=g\alpha'$. We denote the corresponding equivalence class by $[\alpha]$. Composition
and identities are induced by those of $\AtoMon$ and the canonical embedding
$$
    \Gamma\colon
    \AtoMon\to \AtoMon_{\mathrm{reg}}
$$
is given by
$$
    \Gamma(H)=(\operatorname{id}_H\colon H\to H),
    \qquad
    \Gamma(h)=[h].
$$
By \cite[Theorem~8, Proposition~9]{CarboniVitale}, $\AtoMon_{\mathrm{reg}}$ is a regular category and $\Gamma(\AtoMon)$ is a projective cover of $\AtoMon_{\mathrm{reg}}$. The morphisms of this completion can be described using the congruences $\theta_f$.

\begin{proposition}
\label{prop:hom-regular-completion-theta}
Let $f\colon H\to K$ and $g\colon L\to M$ be morphisms in $\AtoMon$. Then
$$
\operatorname{hom}_{\AtoMon_{\mathrm{reg}}}(f,g)
\cong
\left\{
\alpha\colon H\to L:
\theta_f\subseteq\ker_{\mathrm{Mon}}(g\alpha)
\right\}/ \sim
$$
where $\sim$ is the equivalence relation defined by $\alpha\sim\alpha'$ if and only if $g\alpha=g\alpha'$.
\end{proposition}

\begin{proof}
Let $(\ker(f),r_1,r_2)$ be the kernel pair of $f$ in $\AtoMon$ and recall that $\theta_f$ is the smallest monoid congruence containing $\ker(f)$. Thus $\theta_f \subseteq \ker_\Mon(g\alpha)$ if and only if $\ker(f)\subseteq \ker_\Mon(g\alpha)$. This last condition is equivalent to the fact that $g\alpha r_1=g\alpha r_2$.
\end{proof}

We now return to the regular epimorphism $p\colon H\to Q$ of Lemma~\ref{lem:monoregepi}, which determines an object of $\AtoMon_{\mathrm{reg}}$. For the sake of clarity, we denote such object as $\mathbf p=(p\colon H\to Q)$ to distinguish it from $p$ seen in $\AtoMon$.

Notice that, since split epimorphisms are stable under pullbacks, Proposition~\ref{prop:not-regular} implies that $p\colon H \to Q$ is not a split epimorphism.

\begin{proposition}
\label{prop:p-not-in-essential-image}
The object $\mathbf p=(p\colon H\to Q)$ does not belong to the essential
image of
$$
    \Gamma\colon
    \AtoMon\to\AtoMon_{\mathrm{reg}}.
$$
Equivalently, $\mathbf p$ is not isomorphic to $\Gamma(K)$ for any atomic
monoid $K$.
\end{proposition}

\begin{proof}
Suppose that $\mathbf p\cong\Gamma(K)$ for some $K\in\AtoMon$.
Choose inverse morphisms
$$
    [a]\colon\mathbf p\to\Gamma(K),
    \qquad
    [b]\colon\Gamma(K)\to\mathbf p,
$$
represented by morphisms $a\colon H\to K$ and $b\colon K\to H$ in $\AtoMon$.
The equalities $[ab]=[\operatorname{id}_K]$ and $[ba]=\operatorname{id}_\mathbf{p}$ yield
$$
pba=p \qquad \text{and} \qquad ab=\operatorname{id}_K.
$$
Moreover, let $(\ker(p),r_1,r_2)$ be the kernel pair of $p$ in $\AtoMon$. Since $ar_1=ar_2$, there exists a unique morphism $c\colon Q \to K$ such that $a=cp$. Hence, we get $pbcp=p$ and, since $p$ is an epimorphism, $pbc=\operatorname{id}_Q$. This is a contradiction since $p$ is not a split epimorphism (indeed, split epimorphisms are stable under pullbacks, and Proposition~\ref{prop:not-regular} implies that $p\colon H \to Q$ is not a split epimorphism).
\end{proof}

\begin{remark}
The object $\mathbf p$ is therefore a genuinely new object of the regular
completion. It is not represented by an atomic monoid; rather, it is the
image object that must be adjoined in order to obtain a pullback-stable
(regular epi, mono)-factorization of $\Gamma(p)$ in $\AtoMon_\mathrm{reg}$, which becomes
$$
\Gamma(H) \xrightarrow{[\operatorname{id}_H]} \mathbf{p} \xrightarrow{\;[p]\;} \Gamma(Q).
$$
\end{remark}

\section{Adjunctions for natural functors}\label{sec:adjunctions}

In this section, we consider some natural functors among the categories $\Grp, \Mon$ and $\AtoMon$ and establish adjunction between them.

\subsection{Right and left adjoints for the group of units}\label{subsec:grpunits}
Since morphisms in $\AtoMon$ send units into units (being monoid homomorphisms), the assignment $H\to H^\times$ defines a functor
$$
(-)^\times: \AtoMon \to  \mathsf{Grp}.
$$
We show in the next proposition that this functor admits left and right adjoints. 

Let $T:=\{0,a\mid a^2=0\}$ be the semigroup obtained from the terminal
object of $\AtoMon$ by deleting its identity.  Thus
$$
 a^2=a0=0a=0^2=0.
$$
For a group $G$, define
$$
 I(G):=G\oslash T
$$
as the monoid whose underlying set is the disjoint union $G\sqcup\{a,0\}$, and whose operation extends those of $G$ and $T$ by putting
$$
 ga=ag=a,
 \qquad
 g0=0g=0
 \qquad(g\in G).
$$
Then
$$
 I(G)^\times=G,
 \qquad
 \A(I(G))=\{a\},
$$
and $I(G)$ is clearly an atomic monoid. Note that, in the terminology in \cite[Section~3]{CT4}, the monoid $I(G)$ is the {\em trivial ideal extension of $T$ by $G$}.

For a group homomorphism $\varphi:G\to G'$, define
$$
 I(\varphi):I(G)\to I(G')
$$
by $I(\varphi)|_G=\varphi$, $I(\varphi)(a)=a$, and $I(\varphi)(0)=0$.  Clearly $I(\varphi)$ is a monoid homomorphism sending atoms to atoms, so it is a morphism in $\AtoMon$. We therefore obtain a functor
$$
I:\Grp\to\AtoMon.
$$

\begin{proposition}\label{prop:adjunctions}
There is a chain of adjunctions $\iota\dashv(-)^\times\dashv I$, where $\iota:\Grp\to\AtoMon$ is the inclusion functor.
$$
    \xymatrix@!=20pt{
    \AtoMon \ar[rr]_{\;\;\;(-)^\times}^\perp  & & \mathsf{Grp} \ar@/_2.1pc/[ll]_{\iota} \ar@/^2.1pc/[ll]^I_\perp 
    }
$$

\end{proposition}


\begin{proof}
For the adjunction $\iota\dashv(-)^\times$, let $G\in\Grp$ and $H\in\AtoMon$. We want to show that there is a bijection
$$
\hom_\AtoMon(G, H) \to \hom_{\mathsf{Grp}}(G,H^\times),
$$
which is natural both in $G$ and $H$. Every monoid homomorphism $f:G\to H$ has image in $H^\times$, so its corestriction $G \to H^\times$ is a group homomorphism. Conversely, every group homomorphism $G\to H^\times$ becomes a monoid homomorphism $G\to H$ after extending its codomain; preservation of atoms is vacuous because a group has no atoms. This gives the first natural bijection.

For the adjunction $(-)^\times\dashv I$, let $G\in\Grp$ and $H\in\AtoMon$. We want to show that there is a bijection
$$
\hom_{\mathsf{Grp}}(H^\times, G) \to \hom_\AtoMon(H, I(G)),
$$
which is natural both in $H$ and $G$. Any group homomorphism $\varphi \colon H^\times \to G$ extends uniquely to a morphism $\hat{\varphi}\colon H\to I(G)$ of $\AtoMon$ via the assignment
$$
 \hat{\varphi}(x)=
 \begin{cases}
   \varphi(x),&x\in H^\times,\\
   a,&x\in\A(H),\\
   0,&x\in H\setminus(H^\times\cup\A(H)).
 \end{cases}
$$
This defines the required bijection. The inverse assigns to any atom-preserving monoid homomorphism $h\colon H \to I(G)$, the (co)restriction $h|_{H^\times}\colon H^\times \to G$. It is routine to check that this bijection is natural both in $H$ and $G$.
\end{proof}

\subsection{A right adjoint to the underlying-monoid functor}\label{subsec:atomization}

Let us consider the inclusion functor
$$
 U:\AtoMon\to\Mon.
$$
It preserves all small colimits (cf. \cite[Theorem~6.3]{CCT}) and hence it admits a right adjoint by the Adjoint Functor Theorem for locally presentable categories (see \cite[Theorem~1.66]{AR}). In this section, we provide an explicit construction of the right adjoint to $U$.

Let $\mathbb 1:=\{1,a,0\}$ be the terminal object of $\AtoMon$. For a monoid $M$, define the submonoid of $\mathbb 1\times M$
$$
 \At(M):=
 \bigl(\{1\}\times M^\times\bigr)
 \cup
 \bigl(\{a,0\}\times M\bigr)
 \subseteq \mathbb 1\times M,
$$
with componentwise multiplication. It is routine to check that $\At(M)$ is an atomic monoid with
$$
 \At(M)^\times=\{(1,u):u\in M^\times\},
 \qquad
 \A(\At(M))=\{(a,m):m\in M\}.
$$
If $h:M\to N$ is a monoid
homomorphism, set
$$
 \At(h)\colon \At(M) \to \At(N), \qquad \At(h)(t,m):=(t,h(m)).
$$
This is a well-defined monoid homomorphism that sends atoms to atoms. Consequently, $\At$ defines an ``atomization'' functor
$$
\At:\Mon\to\AtoMon.
$$

\begin{proposition}\label{prop:underlying-right-adjoint}
The functor $\At:\Mon\to\AtoMon$ is right adjoint to $U$
$$
\xymatrix{ *+[l]{\AtoMon} \ar@/^7pt/[r]^U \ar@{}[r]|{\perp} & *+[r]{\Mon} \ar@/^7pt/[l]
^{\At}}.
$$
For each monoid $M$, the $M$-component of the counit is the projection
$$
 \varepsilon_M:U(\At(M))\to M,
 \qquad
 \varepsilon_M(t,m):=m.
$$
\end{proposition}

\begin{proof}
We show that there is a natural bijection
$$
 \Hom_{\Mon}(U(H),M)
 \cong
 \Hom_{\AtoMon}(H,\At(M))
$$
for every $H\in\AtoMon$ and every $M\in\Mon$.
Let $f:U(H)\to M$ be a monoid homomorphism.  Define a map $\hat{f}\colon H \to \At(M)$ by
$$
 \hat{f}(x):=
 \begin{cases}
  (1,f(x)) & x\in H^\times,\\
  (a,f(x)) & x\in\A(H),\\
  (0, f(x)) & x\in H\setminus(H^\times\cup\A(H)).
 \end{cases}
$$
This is well-defined since if $x$ is a unit, then $f(x)$ is a unit of $M$. Moreover, $\hat{f}$ is a monoid homomorphism, because atomic monoids are Dedekind-finite, that is, a product $xy \in H$ is a unit if and only if both $x$ and $y$ are units of $H$. Clearly, $\hat{f}$ also preserves atoms, so it is a morphism in $\AtoMon$. This defines a bijection, whose inverse is given, for every $g:H\to \At(M)$, by putting $f:=\varepsilon_M U(g)$. It is routine to check that this bijection is natural in $H$ and $M$.
\end{proof}

\begin{remark}
The functor $\At\colon \Mon \to \AtoMon$ completes the diagram of functors of \cite[Proposition~3.4]{CCT}:
    $$
    \xymatrix@!=40pt{
    \AtoMon \ar[rr]_{\mathscr{A}}^\perp \ar[d]^U_[@][flip][flip]\perp & & \mathsf{Set} \ar@/_1.2pc/[ll]_{\mathscr F} \ar@/^1.3pc/[lld]^{\mathscr F'} \\
    \Mon \ar[rru]^{U'}_[@][flip]\perp \ar@/_-1pc/[u]^\At& & \qquad .
    }
    $$
Here $\mathscr{A}\colon \AtoMon \to \mathsf{Set}$ is the functor sending a monoid to its set of atoms, $U'$ is the forgetful functor from $\Mon$ to $\mathsf{Set}$, and $\mathscr{F}$ (resp., $\mathscr{F'}$) is the free functor from $\mathsf{Set}$ to $\AtoMon$ (resp., $\Mon$).
\end{remark}

\section{Pretorsion theories}\label{sec:pretorsion}

We recall from \cite{FFG} the definition of pretorsion theories.
Let $\Cal C$ be an arbitrary category, and fix a class $\Cal Z$ of objects of $\Cal C$, that we will call the {\em class of trivial objects}. A morphism $f\colon A\to A'$ in $\Cal C$ is said to be \textit{$\Cal Z$-trivial} if it factors through an object of $\Cal Z$. The class of $\Cal Z$-trivial morphisms forms an {\em ideal of morphisms} in $\Cal C$: for every pair of composable morphisms $f$ and $g$ in $\Cal C$, the composition $fg$ is $\Cal Z$-trivial whenever either $f$ or $g$ is $\Cal Z$-trivial. It is possible to introduce the notions of $\mathcal Z$-kernel and $\mathcal Z$-cokernel by replacing, in the usual definitions of kernel and cokernel, the ideal of zero morphisms with the ideal of trivial morphisms induced by the class $\mathcal Z$, as follows.

\begin{definition}
A morphism $\varepsilon\colon X\to A$ in $\Cal C$ is a \emph{$\Cal Z$-kernel} of a morphism $f\colon A \to A'$ if $f\varepsilon$ is $\Cal Z$-trivial and, whenever $\lambda \colon Y\to A$ is a morphism in $\Cal C$ such that $f\lambda$ is $\Cal Z$-trivial, there exists a unique morphism $\lambda'\colon Y\to X$ in $\Cal C$ satisfying $\lambda=\varepsilon\lambda'$.
The notion of \emph{$\Cal Z$-cokernel} is defined dually. A sequence $A\overset{f}{\to}B\overset{g}{\to}C$ is called a \emph{short $\Cal Z$-exact sequence} if $f$ is a $\Cal Z$-kernel of $g$ and $g$ is a $\Cal Z$-cokernel of $f$.
\end{definition}

\begin{definition}\label{def:pretorsion}
Let $\Cal T$ and $\Cal F$ be full and replete subcategories of $\Cal C$.
The pair $(\Cal T,\Cal F)$ is a \emph{pretorsion theory} in $\Cal C$, with class of trivial objects $\Cal Z:=\Cal T\cap \Cal F$, if the following two conditions are satisfied:
\begin{enumerate}[label=\textup{(\roman*)}]
\item every morphism $T\to F$ with $T \in \Cal T$ and $F \in \Cal F$ is $\Cal Z$-trivial;
\item for every object $X$ of $\Cal C$, there exists a short $\Cal Z$-exact sequence
$$\xymatrix{ T_X \ar[r]^f & X \ar[r]^g & F_X}$$
with $T_X\in\Cal T$ and $F_X\in\Cal F$.
\end{enumerate}
If $\Cal C$ is pointed and $\Cal Z$ consists only of the zero objects, then we recover the usual notion of {\em torsion theory} as defined in \cite{D, JT}.
\end{definition}

Recall that $\AtoMon$ is not a pointed category~\cite[Proposition~3.2]{CCT}, hence it cannot admit ``strict'' torsion theories, but only pretorsion theories. In this section we show how to lift torsion theories in $\Grp$ to pretorsion theories in $\AtoMon$. We then observe that this is an instance of a more general construction involving monocoreflective subcategories.

\subsection{Lifting torsion theories from groups to atomic monoids}\label{subsec:pretorsion_atomon}

Let $(\mathcal T,\mathcal F)$ be a torsion theory in $\Grp$. We denote by $t(G)$ the torsion radical of a group $G$, so that, for every $G\in\Grp$, there is a short exact sequence
$$
t(G)\to G\to G/t(G),
$$
where $t(G)\in\mathcal T$ and $G/t(G)\in\mathcal F$.

We regard $\Grp$ as a full subcategory of $\AtoMon$ and define two full
replete subcategories of $\AtoMon$ by
$$
\widetilde{\mathcal T}:=\mathcal T \qquad \text{and}\qquad \widetilde{\mathcal F}
:=
\left\{
H\in\AtoMon:H^\times\in\mathcal F
\right\}.
$$

\begin{proposition}\label{prop:lifting-torsion-theory}
The pair $(\widetilde{\mathcal T},\widetilde{\mathcal F})$ is a pretorsion theory in $\AtoMon$.
\end{proposition}

\begin{proof}
First observe that
$$
\Cal Z:=\widetilde{\mathcal T}\cap\widetilde{\mathcal F}
=
\mathcal T\cap\mathcal F,
$$
so $\Cal Z$ consists of (the isomorphism class of) the trivial monoid $\mathbf 1$, and the $\Cal Z$-trivial morphisms are precisely the morphisms which factor through $\mathbf 1$.

Let $T\in\widetilde{\mathcal T}$ and
$F\in\widetilde{\mathcal F}$. Since $T$ is a group, every morphism
$\varphi:T\to F$ in $\AtoMon$ induces a group homomorphism $\varphi^\times\colon T\to F^\times$, which is trivial because $T\in\mathcal T$ and $F^\times\in\mathcal F$. Hence $\varphi$ factors through $\mathbf 1$.

For $H\in\AtoMon$, let $\theta_H$ be the monoid congruence on $H$ generated by the pairs $\{(u,1_H)\mid u \in t(H^\times)\}$ and let
$$
\rho_H:H\to H/\theta_H
$$
be the canonical projection. It is easy to check that $H/\theta_H$ is atomic and $\rho_H$ is a morphism in $\AtoMon$.
Moreover,
$$
(H/\theta_H)^\times
\cong
H^\times/t(H^\times).
$$
Since $H^\times/t(H^\times)\in\mathcal F$, we have
$$
H/\theta_H\in\widetilde{\mathcal F}.
$$

Consider the following diagram in $\AtoMon$
$$
\xymatrix{
 & t(H^\times) \ar[r]^-{\varepsilon_H} \ar@{=}[d] & H \ar[r]^-{\rho_H} & H/\theta_H & \\
\textbf{1} \ar[r] & t(H^\times) \ar[r] & \; H^\times \ar[r] \ar@{^(->}[u] & \; H^\times /t(H^\times) \ar[r] \ar[u] & \textbf{1},
}
$$
where $\varepsilon_H$ is the inclusion and the second row is a short exact sequence when viewed in $\Grp$. We want to prove that the first row is a short $\Cal Z$-exact sequence in $\AtoMon$.

As $t(H^\times)\in \widetilde{\Cal T}$ and $H/\theta_H\in\widetilde{\mathcal F}$, the composite $\rho_H\varepsilon_H$ is $\Cal Z$-trivial. Let $\lambda:X\to H$ be a morphism in $\AtoMon$ such that $\rho_H\lambda$ is $\Cal Z$-trivial. Then $\rho_H\lambda$ factors through $\mathbf 1$, so there exists a morphism $X\to\mathbf 1$ in $\AtoMon$. This implies that $X$ has no atoms and hence that $X$ is a group. Therefore $\lambda(X)\subseteq H^\times$. Since $\rho_H\lambda$ is trivial and $\ker_{\Grp}(\rho_H^\times)=t(H^\times)$, we obtain $\lambda(X)\subseteq t(H^\times)$. Thus $\lambda$ factors uniquely through $\varepsilon_H$, and hence $\varepsilon_H$ is a $\Cal Z$-kernel of $\rho_H$.

Conversely, let $\mu:H\to Y$ be a morphism such that $\mu \varepsilon_H$ is $\Cal Z$-trivial. Then
$\mu(u)=1_Y$ for every $u\in t(H^\times)$. Hence $\theta_H\subseteq\ker_{\Mon}(\mu)$, and there is a unique monoid homomorphism $\overline{\mu}:H/\theta_H\to Y$ such that $\overline{\mu}\rho_H=\mu$. The map $\overline{\mu}$ preserves atoms. Indeed, every atom of $H/\theta_H$ is of the form $[a]$, with $a\in\mathcal A(H)$, and $\overline{\mu}([a])=\mu(a)\in\A(Y)$.
Thus $\overline{\mu}$ is a morphism in $\AtoMon$, and $\rho_H$ is a $\Cal Z$-cokernel of $\varepsilon_H$.

We have therefore constructed, for every $H\in\AtoMon$, a short $\Cal Z$-exact sequence
$$
t(H^\times)
\to
H
\to
H/\theta_H
$$
whose first term belongs to $\widetilde{\mathcal T}$ and whose last term belongs to $\widetilde{\mathcal F}$. Hence
$(\widetilde{\mathcal T},\widetilde{\mathcal F})$ is a pretorsion theory in $\AtoMon$.
\end{proof}

\begin{example}\label{ex:groups-reduced-pretorsion}
Consider the torsion theory $(\Grp,\{\mathbf 1\})$ in $\Grp$. Its torsion radical is given by $t(G)=G$ for every group $G$. The construction of Proposition~\ref{prop:lifting-torsion-theory} gives
$$
\widetilde{\mathcal T}=\Grp \qquad \text{and} \qquad
\widetilde{\mathcal F}
=
\left\{
H\in\AtoMon:H^\times=\{1_H\}
\right\}.
$$
Thus $\widetilde{\mathcal F}$ is precisely the full subcategory $\RedAtoMon$ of reduced atomic monoids. It follows that
$$
(\Grp,\RedAtoMon)
$$
is a pretorsion theory in $\AtoMon$.
\end{example}

\subsection{Lifting pretorsion theories along a monocoreflection}\label{subsec:pretorsion_general}
The construction of Proposition~\ref{prop:lifting-torsion-theory} is a particular instance of a more general criterion that allows one to lift pretorsion theories along a monocoreflection. In this subsection we describe the general setting, applying the result to a situation different from that of $\AtoMon$.

\medskip

Let $\Cal A\subseteq\Cal B$ be categories and assume that $\Cal A$ is a monocoreflective subcategory of $\Cal B$, that is, the inclusion functor
$
\iota\colon\Cal A\to\Cal B
$
has a right adjoint
$
R\colon\Cal B\to\Cal A
$
and the counit components
$
c_B\colon\iota R(B)\to B
$
are monomorphisms for every $B\in\Cal B$.

Let $(\Cal T,\Cal F)$ be a pretorsion theory in $\Cal A$ with class of trivial objects $\Cal Z$. Let
$$
t\colon\Cal A\to\Cal T \qquad \text{and} \qquad f\colon \Cal A \to \Cal F
$$
be the associated torsion and torsionfree functors. We therefore have, for every object $A \in \Cal A$, a short $\Cal Z$-exact sequence in $\Cal A$:
$$
\xymatrix{
t(A) \ar[r]^-{k_A} & A \ar[r]^-{r_A} & f(A).
}
$$
For ease of readability, we identify $\Cal A$, $\Cal T$ and $\Cal F$ with their essential images in $\Cal B$. In particular, we implicitly assume that the inclusion functors $\iota_\Cal T\colon \Cal T \to \Cal A$, $\iota_\Cal F\colon F \to \Cal A$ and $\iota\colon \Cal A \to \Cal B$ are applied when needed, but we omit to write them explicitly. Define the two full replete subcategories of $\Cal B$
$$
\widetilde{\Cal T}:=\Cal T\quad \text{ and }\quad
\widetilde{\Cal F}
:=
\left\{
B\in\Cal B:R(B)\in\Cal F
\right\}.
$$
Then
$$
\widetilde{\Cal T}\cap\widetilde{\Cal F}=\Cal T\cap\Cal F=\Cal Z.
$$
For every $B\in\Cal B$, set
$$
\varepsilon_B
:=
c_B k_{R(B)}
\colon
tR(B)\to B.
$$
Being a composition of monomorphisms, $\varepsilon_B$ is a monomorphism, and it is precisely the $B$-component of the counit of the adjunction $\iota \circ \iota_\Cal T \dashv t\circ R$. Therefore, the subcategory $\widetilde{\Cal T}=\Cal T$ is monocoreflective in $\Cal B$.

\begin{lemma}\label{lem:lifted-torsion-coreflection}
Every morphism from an object of $\widetilde{\Cal T}$ to an object of $\widetilde{\Cal F}$ is $\Cal Z$-trivial.
\end{lemma}

\begin{proof}
Let $T\in\widetilde{\Cal T}$, $F\in\widetilde{\Cal F}$, and let $\varphi\colon T\to F$ be a morphism in $\Cal B$. Under the adjunction $\iota\dashv R$, the morphism $\varphi$ corresponds to a morphism $\overline \varphi\colon T\to R(F)$ in $\Cal A$. Since $T\in\Cal T$ and $R(F)\in\Cal F$, the morphism $\overline \varphi$ factors through an object in $\Cal Z$, and consequently also $\varphi$ does. Therefore, $\varphi$ is $\Cal Z$-trivial.
\end{proof}

\begin{theorem}\label{thm:lifting-torsion-theory}
With the notation above, assume that the following conditions hold.
\begin{enumerate}[label=\textup{(\arabic*)}]
    \item\label{thm:lifting-torsion-theory 1}
    For every $B \in \Cal B$, if there exists a morphism $B \to Z$ with $Z \in \Cal Z$, then $B \in \Cal A$;
    \item\label{thm:lifting-torsion-theory 2}
    for every $B\in\Cal B$, there exist an object $F_B\in\Cal B$, a morphism $\rho_B\colon B\to F_B$, and an isomorphism $\omega_B\colon R(F_B)\to fR(B)$ such that $\rho_B$ is a $\Cal Z$-cokernel of $\varepsilon_B$ in $\Cal B$ and $\omega_B R(\rho_B)=r_{R(B)}$.
\end{enumerate}
Then, the pair $(\widetilde{\Cal T},\widetilde{\Cal F})$ is a pretorsion theory in $\Cal B$ with class of trivial objects $\Cal Z$.
\end{theorem}

\begin{proof}
We get from Lemma~\ref{lem:lifted-torsion-coreflection} that every morphism from an object of $\widetilde{\Cal T}$ to an object of $\widetilde{\Cal F}$ is $\Cal Z$-trivial.
Let $B \in \Cal B$ and $\rho_B\colon B \to F_B$ be a morphism as in \ref{thm:lifting-torsion-theory 2}. We want to prove that the sequence
$$
\xymatrix{
tR(B) \ar[r]^-{\varepsilon_B} & B \ar[r]^-{\rho_B} & F_B
}
$$
is a short $\Cal Z$-exact sequence in $\Cal B$. By hypothesis, $\rho_B$ is the $\Cal Z$-cokernel of $\varepsilon_B$.
Let $\lambda\colon X \to B$ be a morphism in $\Cal B$ such that $\rho_B\lambda$ is $\Cal Z$-trivial. Then, by \ref{thm:lifting-torsion-theory 1}, $X \in \Cal A$. The morphism $R(\rho_B)R(\lambda)$ is $\Cal Z$-trivial. Moreover, $R(\rho_B)$ is isomorphic to $r_{R(B)}$, so, in particular, $k_{R(B)}$ is the $\Cal Z$-kernel of $R(\rho_B)$. It follows that there exists a (unique) morphism $\sigma \colon R(X)\to tR(B)$ such that $R(\lambda)=k_{R(B)}\sigma$. Since $c_X\colon R(X)\to X$ is an isomorphism, we get a factorization of $\lambda$ through $\varepsilon_B$. This factorization is unique because $\varepsilon_B$ is a monomorphism. This shows that $\varepsilon_B$ is the $\Cal Z$-kernel of $\rho_B$ in $\Cal B$, concluding the proof.
\end{proof}

\begin{remark} Proposition~\ref{prop:lifting-torsion-theory} can be recovered as an application of Theorem~\ref{thm:lifting-torsion-theory}. Indeed, a morphism $H\to\mathbf 1$ in $\AtoMon$ exists only if $H$ is a group, so condition~\ref{thm:lifting-torsion-theory 1} holds. Moreover, the proof of Proposition~\ref{prop:lifting-torsion-theory} shows that $\rho_H$ is the $\Cal Z$-cokernel of $\varepsilon_H$, while the canonical isomorphism
$(H/\theta_H)^\times\cong H^\times/t(H^\times)$ identifies $R(\rho_H)=\rho_H^\times$ with the torsionfree reflection of $H^\times$. Thus condition~\ref{thm:lifting-torsion-theory 2} is also satisfied.
\end{remark}

For the sake of completeness, we apply the criterion of Theorem~\ref{thm:lifting-torsion-theory} to a a context different from that of monoids. In the following example, we lift a torsion theory in the category of pointed sets to a pretorsion theory in the category of pointed quivers.

\begin{example}\label{ex:pointed-quivers}
Let $\Quiv_*$ be the category of pointed quivers. An object of $\Quiv_*$ consists of a quiver
$$
Q\colon E(Q)\overset{s}{\underset{t}{\rightrightarrows}}V(Q)
$$
together with a distinguished vertex $v_Q\in V(Q)$, and morphisms are
quiver morphisms preserving the distinguished vertex. 
Let $\Set_*$ be the category of pointed sets, regarded as the full
subcategory of $\Quiv_*$ consisting of discrete pointed quivers. There is an adjunction $D\dashv V$, where
$$
D\colon\Set_*\to\Quiv_*
$$
is the inclusion functor (every pointed set is regarded as a discrete quiver) and 
$$
V\colon\Quiv_*\to\Set_*
$$
is the functor sending a pointed quiver to its pointed set of vertices. Indeed, a morphism from a discrete pointed quiver $D(X)$ to a pointed quiver $Q$ is uniquely determined by a pointed map $X\to V(Q)$.
The counit
$$
c_Q\colon D(V(Q))\to Q
$$
is the identity on vertices and the unique map $\emptyset\to E(Q)$ on edges. In particular, it is a monomorphism. Thus $\Set_*$ is a monocoreflective subcategory of $\Quiv_*$.

Consider in $\Set_*$ the torsion theory
$$
(\Cal T,\Cal F)=(\Set_*,\Cal Z),
$$
where $\Cal Z$ denotes the full subcategory of one-pointed sets. Its associated functors
are
$$
t=\id_{\Set_*}
\qquad\text{and}\qquad
f=f_0,
$$
where $f_0$ sends a pointed set $(X,x_0)$ to $(\{x_0\}, x_0)$.
The induced subcategories of $\Quiv_*$ are
$$
\widetilde{\Cal T}
=
\{\text{discrete pointed quivers}\}
 \qquad \text{and} \qquad 
\widetilde{\Cal F}
=
\{\text{pointed quivers having exactly one vertex}\}.
$$
Their intersection is the full replete subcategory of pointed quivers with one vertex and no edges.

We verify the hypotheses of Theorem~\ref{thm:lifting-torsion-theory}.
First, a quiver morphism $Q\to Z$, with $Z \in \Cal Z$, exists if and only if $E(Q)=\emptyset$. Thus, if such a morphism exists, then $Q$ is discrete and hence belongs to $\Set_*$. Thus condition \ref{thm:lifting-torsion-theory 1} holds.

Now let $Q\in\Quiv_*$ and let $v_Q$ be its distinguished vertex. Define a pointed quiver $F_Q$ by
$$
V(F_Q)=\{v_Q\},
\qquad
E(F_Q)=E(Q),\qquad \text{and}\qquad s(e)=t(e)=v_Q \; \text{ for every } e \in E(Q).
$$
Thus every edge of $F_Q$ is a loop at its unique
vertex. Define
$$
\rho_Q\colon Q\to F_Q
$$
to be the constant map on vertices and the identity map on edges. Since $F_Q$ has exactly one vertex, we have $F_Q\in\widetilde{\Cal F}$. We claim that $\rho_Q$ is the $\Cal Z$-cokernel of $\varepsilon_Q=c_Q\colon D(V(Q))\to Q$. The composite $\rho_Q\varepsilon_Q$ factors through $\Cal Z$, since it sends
every vertex of $Q$ to the unique vertex $v_Q$ of $F_Q$.

Let $h\colon Q\to Y$ be a morphism such that $h\varepsilon_Q$ is $\Cal Z$-trivial. This is
equivalent to saying that the vertex map $h_V\colon V(Q)\to V(Y)$ sends every vertex of $Q$ to the distinguished vertex $v_Y$ of $Y$.
Define
$$
\overline h\colon F_Q\to Y
$$
by sending the unique vertex of $F_Q$ to $v_Y$ and by setting $\overline h_E(e):=h_E(e)$ for every $e\in E(Q)$. Since $h_V$ is constant at $v_Y$, for every edge $e\in E(Q)$ one has $s(h_E(e))=t(h_E(e))=v_Y$.
Therefore, $\overline h$ is a well-defined morphism of pointed quivers, and $\overline h\rho_Q=h$.
The factorization is unique because $\rho_Q$ is surjective on vertices and is the identity on edges. Hence $\rho_Q$ is the $\Cal Z$-cokernel of $\varepsilon_Q$.

Finally,
$$
V(\rho_Q)\colon V(Q)\to V(F_Q)
$$
is the unique pointed map to $\{v_Q\}$. It is therefore precisely the torsionfree reflection
$$
r_{V(Q)}\colon V(Q)\to f(V(Q)).
$$
All the hypotheses of Theorem~\ref{thm:lifting-torsion-theory} are
satisfied. Consequently, $(\widetilde{\Cal T},\widetilde{\Cal F})$ is a pretorsion theory in $\Quiv_*$. For every pointed quiver $Q$, its associated short $\Cal Z$-exact sequence is
$$
\xymatrix{
D(V(Q)) \ar[r]^-{c_Q} & Q \ar[r]^{\rho_Q} & F_Q.
}
$$
\end{example}

\end{document}